\documentclass[12pt,oneside,reqno]{amsart}
\usepackage{geometry}
\usepackage{float}
\usepackage{amsmath, amssymb, amsthm}
\usepackage[utf8]{inputenc}
\usepackage{graphicx}
\usepackage[pdftex]{hyperref}
\usepackage{listings}
\usepackage{indentfirst}
\usepackage{theoremref}
\usepackage{thmtools}
\usepackage{physics}
\usepackage{mathtools}

\allowdisplaybreaks[3]

\newcommand{\innprod}[2]{\left< #1, #2 \right>}
\newcommand{\R}{\mathbb{R}}
\newcommand{\N}{\mathbb{N}}
\renewcommand{\complement}[1]{#1^\mathrm{C}}
\DeclareMathOperator{\Hess}{Hess}
\DeclareMathOperator{\lin}{lin}
\DeclareMathOperator{\Dom}{Dom}

\newtheorem{theorem}{Theorem}
\newtheorem{lemma}[theorem]{Lemma}
\newtheorem{proposition}[theorem]{Proposition}

\theoremstyle{definition}

\theoremstyle{remark}

\numberwithin{equation}{section}

\begin{document}
	
	\title[Dimension-free estimates for Riesz transforms]{Dimension-free estimates for Riesz transforms related to the harmonic oscillator}
	
	\author{Maciej Kucharski}
	
	\address{Maciej Kucharski\\
		Instytut Matematyczny\\
		Uniwersytet {Wrocławski}\\
		Plac Grun\-waldzki 2/4\\
		50-384 {Wrocław}\\
		Poland}
	\email{mkuchar@math.uni.wroc.pl}
	
	
	
	\begin{abstract}
		We study $L^p$ bounds for two kinds of Riesz transforms on $\R^d$ related to the harmonic oscillator. We pursue an explicit estimate of their $L^p$ norms that is independent of the dimension $d$ and linear in $\max(p, p/(p-1))$.
	\end{abstract}
	
	\maketitle
	
	\section{Introduction}
	
	The aim of this paper is to prove a dimension-free estimate for the $L^p$ norm of vectors of a specific kind of generalized Riesz transforms. Recall that the classical Riesz transforms on $\R^d$ are the operators
	\[
	R_if(x) = \partial_{x_i} \left(-\Delta \right)^{-1/2}f(x), \quad i=1, \dots, d.
	\]
	A well-known result concerning Riesz transforms, proved by Stein in \cite{stein}, is the $L^p$ boundedness of the vector of the Riesz transforms
	\[
	\mathbf{R} f = \left( R_1f, \dots, R_d f \right)
	\]
	with a norm estimate independent of $d$. Since then, the question about dimension-free estimates for the Riesz transforms has been asked in various contexts. For example Carbonaro and Dragi\v{c}ević proved in \cite{carbonaro_dragicevic} a dimension-free estimate with an explicit constant for the shifted Riesz transform on a complete Riemannian manifold. Another path of generalizing the result of Stein is to consider operators of the form
	\begin{equation} \label{riesz}
	R_i = \delta_i L^{-1/2},
	\end{equation}
	where $\delta_i$ is an operator on $L^2(\R^d)$ and
	\begin{equation} \label{eq:L2}
	L = \sum_{i=1}^d L_i = \sum_{i=1}^d \left( \delta_i^* \delta_i + a_i \right), \quad a_i \geqslant 0.
	\end{equation}
	Such Riesz transforms were studied systematically by Nowak and Stempak in \cite{nowak_stempak}. We will focus on the Riesz transforms of the form as in \eqref{riesz} where 
	\begin{equation} \label{eq:L}
	L = -\Delta + \abs{x}^2
	\end{equation}
	is the harmonic oscillator, i.e.
	\begin{equation} \label{eq:delta}
	\delta_i = \partial_{x_i} + x_i, \quad \delta_i^* = -\partial_{x_i} + x_i, \quad a_i = 1,
	\end{equation}
	which means that
	\begin{equation} \label{eq:delta2}
	\begin{aligned}
	\delta_i \delta_i^* &= -\partial_{x_i x_i}^2 + (x_i^2 + 1) \\
	\delta_i^* \delta_i &= -\partial_{x_i x_i}^2 + (x_i^2 - 1). \\
	\end{aligned}
	\end{equation}
	From this point $\delta_i$ and $\delta_i^*$ are defined as above.
	
	This so-called \emph{Hermite-Riesz transform} was introduced by Thangavelu in \cite{thangavelu}, who proved its $L^p$ boundedness. Then a dimension-free estimate of its norm was proved in \cite{harboure} and \cite{lust_piquard}, which later was sharpened by Dragi\v{c}ević and Volberg in \cite{dragicevic_volberg} to an estimate linear in $\max(p, p/(p-1))$. 
	
	In the first part we will give a result analogous to \cite[Theorem 9]{wrobel}, however concerning a slightly altered operator, namely
	\[
	R_i' = \delta_i^* L'^{-1/2}
	\]
	with
	\[
	L_i' = \delta_i\delta_i^* + 1, \quad L' = \sum_{i=1}^d L_i'.
	\]
	It arises as a result of swapping $\delta_i$ and $\delta_i^*$ in formulae \eqref{riesz} and \eqref{eq:L2}. As explained in Section \ref{sec3}, the results from \cite{wrobel} do not apply to this operator. The key step in the proof is, as in \cite{wrobel}, the method of Bellman function but we use its more subtle properties to achieve the goal.
	
	In the second part we consider the vector of the Riesz transforms
	\[
	{\mathbf{\tilde{R}}} f = \left( \tilde{R}_1 f, \dots, \tilde{R}_d f \right),
	\]
	where
	\[
	\tilde{R}_i = \delta_i^*L^{-1/2}.
	\]
	Its boundedness was proved in \cite{dragicevic_volberg} (where $\tilde{R}_i$ was denoted by $R_i^*$), \cite{harboure} and \cite{lust_piquard} with an implicit constant independent of the dimension. Our goal is to give an explicit constant. Due to reasons explained in Section \ref{sec4} we will focus on proving the boundedness of the operator $S$ defined as
	\[
	Sf(x) = \abs{x} L^{-1/2}f(x).
	\]
	We obtain it by an explicit estimate of the kernel of $S$. As a corollary we get a dimension-free estimate of the norm of the vector of the operators
	\[
	R_i^* = \delta_i^* (L+2)^{-1/2}
	\]
	with each $R_i^*$ being the adjoint of $R_i = \delta_i L^{-1/2}$ studied in \cite{dragicevic_volberg} and \cite{wrobel}.

	\section{Preliminaries}
	For the sake of simplicity we consider only real-valued functions. All results can be extended to complex-valued functions with twice as large constants.
	
	In order to define the operators $L'$, $L$, $R_i'$ and $\tilde{R}_i$ on $L^2(\R^d)$ (later abbreviated as $L^2$) we introduce the Hermite polynomials and the Hermite functions. The Hermite polynomials are given by
	\[
	H_n(x) = (-1)^n e^{x^2} \dv[n]{x} e^{-x^2}, \ x \in \R
	\]
	or, equivalently, by
	\[
	H_n(x) = 2 x H_{n-1}(x) - 2 (n-1) H_{n-2}(x), \quad n \geqslant 2, \ x \in \R,
	\]
	\[
	H_0(x) = 1, \ H_1(x) = 2x.
	\]
	The Hermite functions are
	\[
	h_n(x) =  \frac{1}{\sqrt{2^n n! \sqrt{\pi}}} e^{-x^2/2} H_n(x), \ x \in \R.
	\]
	It is well known that the Hermite functions form an orthonormal basis of $L^2(\R)$ and that their linear span is dense in $L^p(\R)$ for every $1 \leqslant p < \infty$.
	
	For $n = (n_1, \dots, n_d) \in \N^d$ with $\N = \{0, 1, 2 \dots \}$ and $x = (x_1, \dots, x_d) \in \R^d$ we define
	\[
	h_n(x) = h_{n_1}(x_1) \cdots h_{n_d}(x_d).
	\]
	We can see that $\{ h_n \}_{n \in \N^d}$ is an orthonormal basis of $L^2$. Throughout the paper we will use 
	\[
	\mathcal{D} = \lin \{ h_n: n \in \N^d \}.
	\]
	
	Recalling \eqref{eq:delta}, on $C_c^\infty(\R^d)$ we define
	\begin{align*}
	L_i = \delta_i^* \delta_i + 1 \qquad &\text{and} \qquad L = \sum_{i=1}^d L_i \\
	L_i' = \delta_i \delta_i^* + 1 \qquad &\text{and} \qquad L' = \sum_{i=1}^d L_i'.
	\end{align*}
	We remind the reader that $\delta_i^* \delta_i$ and $\delta_i \delta_i^*$ were calculated in \eqref{eq:delta2}. Since $\delta_i \delta_i^* = \delta_i^* \delta_i + 2$, we can also write
	\begin{equation} \label{eq:l'}
	L' = L+2d.
	\end{equation}
	
	Note that the formal adjoint of $\delta_i$ with respect to the inner product on $L^2$ is $\delta_i^* = -\partial_{x_i} + x_i$. We recall well-known relations concerning the Hermite functions.
	\begin{lemma} \thlabel{fact}
		For $n \in \N^d$ and $i = 1, \dots, d$ we have
		\begin{enumerate}
			\item $\delta_i h_n = 
			\begin{cases}
			\sqrt{2n_i} h_{n-e_i} \quad &\text{if } n_i \neq 0 \\
			0 &\text{otherwise}
			\end{cases}$ \label{eq:fact1}
			\item $\delta_i^* h_n = \sqrt{2(n_i+1)} h_{n+e_i}$, \label{eq:fact2}
			\item $L'_i h_n = (2n_i+3) h_n$, \label{eq:fact3}
			\item $L_i  h_n = (2n_i+1) h_n$. \label{eq:fact4}
		\end{enumerate}
	\end{lemma}
	Hence, the multivariate Hermite functions $\{h_n\}_{n \in \N^d}$ are eigenvectors of $L'$ and $L$ corresponding to positive eigenvalues $\{\lambda_n'\}_{n \in \N^d}$ and $\{\lambda_n\}_{n \in \N^d}$ respectively, namely
	\begin{equation} \label{eq:eigen}
	\begin{aligned}
	Lh_n = \lambda_n h_n \qquad &\text{for} \qquad \lambda_n = 2\abs{n}+d \\
	L' h_n = \lambda_n' h_n \qquad &\text{for} \qquad \lambda_n' = 2\abs{n}+3d
	\end{aligned}
	\end{equation}
	By a well-known result due to Kato (see \cite[p. 137]{kato}) the operators $L$ (and $L'$) are essentially self-adjoint on $C_c^\infty(\R^d)$ with the self-adjoint extensions given by
	\[
	Lf = \sum_{n \in \N^d} \lambda_n \innprod{f}{h_n} h_n \qquad \text{and}
	\qquad
	L'f = \sum_{n \in \N^d} \lambda_n' \innprod{f}{h_n} h_n,
	\]
	where $\innprod{\cdot}{\cdot}$ denotes the $L^2$ inner product, acting on the domains
	\[
	\Dom(L) = \{ f \in L^2 : \sum_{n \in \N^d} \lambda_n^2 \abs{\innprod{f}{h_n}}^2 < \infty \},
	\]
	\[
	\Dom(L') = \{ f \in L^2 : \sum_{n \in \N^d} \lambda_n'^2 \abs{\innprod{f}{h_n}}^2 < \infty \}.
	\]
	Then $R_i' = \delta_i^* L'^{-1/2}$ can be defined rigorously as
	\[
	R_i' f = \sum_{n \in \N^d} \lambda_n'^{-1/2} \innprod{f}{h_n} \delta_i^* h_n
	\]
	and $\tilde{R}_i = \delta_i^*L^{-1/2}$ as
	\[
	\tilde{R}_i f = \sum_{n \in \N^d} \lambda_n^{-1/2} \innprod{f}{h_n} \delta_i^* h_n.
	\]
	It is clear that $R_i'$ and $\tilde{R}_i$ are bounded on $L^2$.
	
	In what follows we will often identify a densely defined bounded operator on a Banach space with its unique bounded extension to the whole space. As for the notation, we will abbreviate
	\[
	L^p = L^p(\R^d), \quad \norm{\cdot}_p = \norm{\cdot}_{L^p} \quad \text{and} \quad \norm{\cdot}_{p \rightarrow p} = \norm{\cdot}_{L^p \rightarrow L^p}
	\]
	and for $x = \left(x_1, \dots, x_d \right) \in \R^d$ we will use $\abs{x} = \left( \sum_{i=1}^d x_i^2 \right)^{1/2}$. For $1 < p < \infty$ we denote $p^* = \max \left(p, \frac{p}{p-1} \right)$.
	
	\section{Riesz transforms of the first kind} \label{sec3}
	
	Let $\mathbf{R'} f = \left( R_1' f, \dots, R_d' f \right)$. The main result of this section is an explicit estimate for the $L^p$ norm of $\mathbf{R'}$.
	
	\begin{theorem} \thlabel{th1}
		For $f \in L^p$ we have
		\[
		\norm{\mathbf{R'} f}_p \coloneqq \left( \int_{\R^d} \abs{\mathbf{R'}f(x)}^p \, dx \right)^{1/p} \leqslant 36 (p^* - 1) \norm{f}_p.
		\]
	\end{theorem}
	
	In order to prove \thref{th1}, we will need some auxiliary objects. One can see that $L_i' = -\partial_{x_i}^2 + x_i^2 + 2$, so we can write
	\begin{equation} \label{eq:r}
	-\Delta = -\sum_{i=1}^d \partial_{x_i}^2 = L' - r, \quad \text{ where } r(x) = \abs{x}^2 + 2d.
	\end{equation}
	
	We will also need the operator $M$ defined on $C_c^\infty(\R^d)$ as
	\[
	M = L' + [\delta_i^*, \delta_i] = L' - 2,
	\]
	where
	\[
	[\delta_i^*, \delta_i] = \delta_i^* \delta_i - \delta_i \delta_i^*.
	\]
	Note that in our case $[\delta_i^*, \delta_i] = -2 < 0$. This means that the crucial assumption from \cite{wrobel} does not hold and the theory does not apply.
	
	Having these operators, we can introduce the semigroups
	\[
	P_t = e^{-t L'^{1/2}} \quad \text{ and } \quad Q_t = e^{-t M^{1/2}}
	\]
	rigorously defined as
	\begin{equation} \label{eq:semigroups}
	P_t f = \sum_{n \in \N^d} e^{-t \lambda_n'^{1/2}} \innprod{f}{h_n} h_n, \quad Q_t f = \sum_{n \in \N^d} e^{-t (\lambda_n'-2)^{1/2}} \innprod{f}{h_n} h_n.
	\end{equation}
	For $f = (f_1, \dots, f_d)$ we write
	\[
		Q_t f = \left( Q_t f_1, \dots, Q_t f_d \right).
	\]
	
	\begin{lemma} \thlabel{lem2}
		Let $i = 1, \dots, d$. If $f, g \in \mathcal{D}$, then
		\begin{equation} \label{eq:lem3}
		\innprod{R_i' f}{g} = -4 \int_0^\infty \innprod{\delta_i^* P_t f}{\partial_t Q_t g} t \, dt.
		\end{equation}
	\end{lemma}
	
	\begin{proof}
		The proof is similar to the proof of \cite[Proposition 2]{wrobel} but we give it for the sake of completeness. By linearity it is sufficient to prove the lemma for $f = h_n$ and $g = h_k$ for some $n,k \in \N^d$. We proceed as follows:
		\begin{align*}
		-4 \int_0^\infty \innprod{\delta_i^* P_t h_n}{\partial_t Q_t h_k} \, t \, dt &= -4 \int_0^\infty \innprod{e^{-t \lambda_n'^{1/2}} \delta_i^* h_n}{-(\lambda_k'-2)^{1/2} e^{-t(\lambda_k'-2)^{1/2}} h_k} t \, dt \\
		&= 4 (\lambda_k'-2)^{1/2} \innprod{\delta_i^* h_n}{h_k} \int_0^\infty e^{-t (\lambda_n'^{1/2} + (\lambda_k'-2)^{1/2})} \, t \, dt \\
		&= \frac{4(\lambda_k'-2)^{1/2}}{\left( \lambda_n'^{1/2} + (\lambda_k'-2)^{1/2} \right)^2} \innprod{\delta_i^* h_n}{h_k}.
		\end{align*}
		Hence, we get
		\begin{align*}
		&\innprod{\delta_i^* L'^{-1/2} h_n}{h_k} + 4 \int_0^\infty \innprod{\delta_i^* P_t h_n}{\partial_t Q_t h_k} t \, dt \\
		&= \lambda_n'^{-1/2} \innprod{\delta_i^* h_n}{h_k} - \frac{4(\lambda_k'-2)^{1/2}}{\left( \lambda_n'^{1/2} + (\lambda_k'-2)^{1/2} \right)^2} \innprod{\delta_i^* h_n}{h_k} \\
		&= \left( \lambda_n'^{-1/2} - \frac{4(\lambda_k'-2)^{1/2}}{\left( \lambda_n'^{1/2} + (\lambda_k'-2)^{1/2} \right)^2} \right) \innprod{\delta_i^* h_n}{h_k} \\
		&= \sqrt{2(n_i+1)}\left( \lambda_n'^{-1/2} - \frac{4(\lambda_k'-2)^{1/2}}{\left( \lambda_n'^{1/2} + (\lambda_k'-2)^{1/2} \right)^2} \right) \innprod{h_{n+e_i}}{h_k}.
		\end{align*}
		In the last equality we used item \ref{eq:fact2}. of \thref{fact}.
		If $k = n+e_i$, then by \eqref{eq:eigen} we have $\lambda_k' = \lambda_n' + 2$ and the expression in parentheses is 0, otherwise $h_{n+e_i}$ and $h_k$ are orthogonal.
		
	\end{proof}
	
	We will also need a bilinear embedding theorem. Before stating it, for $f = \left(f_1, \dots, f_N \right): \R^d \times (0, \infty) \rightarrow \R^N$ we set
	\begin{align*}
	\abs{f(x,t)}_*^2 &= r(x) \abs{(f_1(x,t), \dots, f_N(x,t))}^2 \\
	&+ \abs{(\partial_t f_1(x,t), \dots, \partial_t f_N(x,t))}^2 \\
	&+ \sum_{i=1}^d \abs{(\partial_{x_i} f_1(x,t), \dots, \partial_{x_i} f_N(x,t))}^2.
	\end{align*}
	
	\begin{theorem} \thlabel{bet}
		Let $d \geqslant 2$, $1 < p < \infty$ and $q = \frac{p}{p-1}$. If $f \in \mathcal{D}$ and $g = \left( g_1, \dots, g_d \right)$ with $g_i \in \mathcal{D}$ for $i = 1, \dots, d$, then we have
		\[
		\int_0^\infty \int_{\R^d} \abs{P_t f(x)}_* \abs{Q_t g(x)}_* \, dx \, t \, dt \leqslant 6(p^*-1) \norm{f}_p\norm{g}_q.
		\]
	\end{theorem}
	
	\subsection{The Bellman function} \label{sec31}
	
	In order to prove \thref{bet}, let us introduce the Bellman function. Take $p \geqslant 2$ and let $q$ be its conjugate exponent. Define $\beta: [0, \infty)^2 \rightarrow [0, \infty)$ by
	\begin{equation} \label{eq:beta}
	\beta(s,t) = s^p + t^q + \gamma
	\begin{cases}
	s^2 t^{2-q} \quad &\text{if } s^p \leqslant t^q \\
	\frac{2}{p} s^p + \left( \frac{2}{q} - 1 \right) t^q \quad &\text{if } s^p \geqslant t^q
	\end{cases}, 
	\quad \gamma = \frac{q(q-1)}{8}.
	\end{equation}
	The Nazarov--Treil Bellman function is then the function
	\[
	B(\zeta,\eta) = \tfrac{1}{2} \beta \left( \abs{\zeta}, \abs{\eta} \right), \quad \zeta \in \R^{m_1}, \eta \in \R^{m_2}.
	\]
	It was introduced by Nazarov and Treil in \cite{nazarov_treil} and then simplified and used by Carbonaro and Dragi\v{c}ević in \cite{carbonaro_dragicevic,carbonaro_dragicevic2} and by Dragi\v{c}ević and Volberg in \cite{dragicevic_volberg3,dragicevic_volberg4,dragicevic_volberg}.
	Note that $B$ is differentiable but not smooth, so we convolve it with a mollifier $\psi_\kappa$ to get $B_\kappa = B * \psi_\kappa$, where
	\[
	\psi_\kappa(x) = \frac{1}{\kappa^{m_1+m_2}} \psi \left( \frac{x}{\kappa} \right) \quad \text{ and } \quad \psi(x) = c_{m_1,m_2} e^{-\frac{1}{1-\abs{x}^2}} \chi_{B(0,1)}(x), \quad x \in \R^{m_1+m_2}
	\]
	and $c_{m_1,m_2}$ is the $L^1$-normalizing constant. The functions $B$ and $\psi_\kappa$ are biradial and so is $B_\kappa$, hence there exists $\beta_\kappa: [0, \infty)^2 \rightarrow [0, \infty)$ such that
	\[
	B_\kappa(\zeta,\eta) = \tfrac{1}{2} \beta_\kappa \left( \abs{\zeta}, \abs{\eta} \right).
	\]
	We invoke some properties of $\beta_\kappa$ and $B_\kappa$ that were proved in \cite{dragicevic_volberg} and \cite[Theorem 4]{carbonaro_dragicevic}.
	
	\begin{theorem} \thlabel{bellman}
		Let $\kappa \in (0,1)$ and $s,t > 0$. Then we have
		\begin{enumerate}
			\item $0 \leqslant \beta_\kappa(s,t) \leqslant \left( 1 + \gamma \right) \left((s+\kappa)^p + (t+\kappa)^q \right) \label{eq:bellman1}$,
			\item $0 \leqslant \partial_s \beta_\kappa(s,t) \leqslant C_p \max \left( (s+\kappa)^{p-1}, t+\kappa \right)$,
			\item [\hspace{2ex}] $0 \leqslant \partial_t \beta_\kappa(s,t) \leqslant C_p (t+\kappa)^{q-1} \label{eq:bellman2}$.
		\end{enumerate}
		The function $B_\kappa$ is smooth and for every $z = (x,y) \in \R^{m_1+m_2}$ there exists $\tau_\kappa > 0$ such that for $\omega = (\omega_1, \omega_2) \in \R^{m_1+m_2}$ we have
		\begin{enumerate}
			\setcounter{enumi}{2}
			\item $\innprod{\Hess(B_\kappa)(z)\omega}{\omega} \geqslant \frac{\gamma}{2} \left(  \tau_\kappa \abs{\omega_1}^2 + \tau_\kappa^{-1} \abs{\omega_2}^2 \right). \label{eq:bellman3}$
		\end{enumerate}
		There is a continuous function $E_\kappa: \R^{m_1+m_2} \rightarrow \R$ such that
		\begin{enumerate}
			\setcounter{enumi}{3}
			\item $\innprod{\nabla B_\kappa (z)}{z} \geqslant \frac{\gamma}{2} \left(  \tau_\kappa \abs{x}^2 + \tau_\kappa^{-1} \abs{y}^2 \right) - \kappa E_\kappa(z) + B_\kappa(z) \label{eq:bellman4}$,
			\item $\abs{E_\kappa(z)} \leqslant C_{m_1,m_2,p} \left( \abs{x}^{p-1} + \abs{y} + \abs{y}^{q-1} + \kappa^{q-1} \right) \label{eq:bellman5}$.
		\end{enumerate}
	\end{theorem}
	
	\subsection{Proof of \texorpdfstring{\thref{bet}}{Theorem 4}}
	
	Having defined the Bellman function, we proceed to the proof. First we should emphasize that the presence of the term $B_\kappa(z)$ in \ref{eq:bellman4}. is the key ingredient for the Bellman method to work despite the fact that $[\delta_i^*, \delta_i] < 0$. Because of that, the proof of \thref{liminf} is more involved than in \cite{wrobel}.
	
	Let
	\begin{equation} \label{eq:u}
	u(x,t) = \left( P_tf(x), Q_tg(x) \right) = \left( P_tf(x), Q_t g_1(x), \dots, Q_t g_d(x) \right)
	\end{equation}
	for $x \in \R^d$ and $t > 0$ and fix $p \geqslant 2$. We will use the Bellman function $B_\kappa$ and $b_\kappa = B_\kappa \circ u$ with $m_1 = 1$ and $m_2 = d$. Our aim is to estimate the integral
	\[
	I(n,\varepsilon) = \int_0^\infty \int_{X_n} \left( \partial_t^2 + \Delta \right) (b_{\kappa(n)})(x,t) \, dx \, t e^{-\varepsilon t} \, dt,
	\]
	where $\kappa(n)$ is a number depending on $n$ and $X_n = [-n, n]^d$ so that $\{X_n\}_{n \in \N}$ is an increasing family of compact sets such that $\R^d = \bigcup_n X_n$.
	
	\begin{lemma} \thlabel{liminf}
		We have
		\[
		\liminf_{\varepsilon \rightarrow 0^+} \liminf_{n \rightarrow \infty} I(n,\varepsilon) \geqslant \gamma \int_0^\infty \int_{\R^d} \abs{P_t f(x)}_* \abs{Q_t g(x)}_* \, dx \, t \, dt.
		\]
		
		%\left( \partial_t^2 - \tilde{L} \right) (b_\kappa) (x,t) \geqslant \gamma \abs{F(x,t)}_* \abs{G(x,t)}_* - \kappa r(x) E_\kappa(u(x,t)),
		
	\end{lemma}
	
	\begin{proof}
		In order to make formulae more compact, we will sometimes write $\partial_{x_0}$ instead of $\partial_t$. The first step will be to prove that
		\begin{equation} \label{eq:1}
		\begin{aligned} 
		\left(\partial_t^2 + \Delta \right) (b_\kappa) (x,t) &\geqslant \gamma \abs{P_t f(x)}_* \abs{Q_t g(x)}_* - \kappa r(x) E_\kappa(u(x,t)) \\
		&+ r(x)B_\kappa(u(x,t)) - 2 \sum_{i=1}^d \partial_{\eta_i} B_\kappa(u(x,t)) Q_t g_i(x).
		\end{aligned}
		\end{equation}
		Recall that $r$ was defined in \eqref{eq:r}. From the chain rule we get $\partial_{x_i} b_\kappa (x,t) = \innprod{\nabla B_\kappa(u(x,t))}{\partial_{x_i} u(x,t)}$ for $i = 0, \dots, d$. Then, again by the chain rule, we have
		\[
		\partial_{x_i}^2 b_\kappa (x,t) = \innprod{\nabla B_\kappa(u(x,t))}{\partial_{x_i}^2 u(x,t)} + \innprod{\Hess(B_\kappa)(u(x,t))(\partial_{x_i} u(x,t))}{\partial_{x_i} u(x,t)}.
		\]
		Summing for $i = 0, \dots, d$, we get
		\begin{align*}
		\left(\partial_t^2 + \Delta \right) (b_\kappa) (x,t) &= \innprod{\nabla B_\kappa(u(x,t))}{(\partial_t^2 + \Delta) (u)(x,t)} \\
		&+ \sum_{i=0}^d \innprod{\Hess(B_\kappa)(u(x,t))(\partial_{x_i} u(x,t))}{\partial_{x_i} u(x,t)}.
		\end{align*}
		
		By the definition of $P_t$ and $Q_t$ we see that
		\[
		(\partial_t^2 - L') P_t f = 0
		\]
		and
		\[
		(\partial_t^2 - L') Q_t g_i = (\partial_t^2 - M) Q_t g_i - 2 Q_t g_i = - 2 Q_t g_i,
		\]
		which can be summarized as
		\[
		\left( \partial_t^2 + \Delta \right)(u) = ru - 2\left(0, Q_t g \right).
		\]
		Therefore, using the fact that $-\Delta = L'-r$ we get
		\begin{align*}
		\left(\partial_t^2 + \Delta \right) (b_\kappa) (x,t) &= r(x) \innprod{\nabla B_\kappa (u(x,t))}{u(x,t)} \\
		&- 2 \sum_{i=1}^d \partial_{\eta_i} B_\kappa(u(x,t)) Q_t g_i(x) \\
		&+ \sum_{i=0}^d \innprod{\Hess(B_\kappa)(u(x,t))\left(\partial_{x_i} u(x,t) \right)}{\partial_{x_i} u(x,t)}.
		\end{align*}
		
		Next, inequalities \ref{eq:bellman3}. and \ref{eq:bellman4}. from \thref{bellman} and the inequality of arithmetic and geometric means imply that
		\begin{align*}
		\left(\partial_t^2 + \Delta \right) (b_\kappa) (x,t) \geqslant & \, r(x)\frac{\gamma}{2} \left( \tau_\kappa \abs{P_t f(x)}^2 + {\tau_\kappa}^{-1} \abs{Q_t g(x)}^2 \right) \\
		&- r(x)\kappa E_\kappa(u(x,t)) + r(x)B_\kappa(u(x,t)) \\
		&- 2 \sum_{i=1}^d \partial_{\eta_i} B_\kappa(u(x,t)) Q_t g_i(x) \\
		&+ \frac{\gamma}{2} \sum_{i=0}^d \left( \tau_\kappa \abs{\partial_{x_i} P_t f(x)}^2 + {\tau_\kappa}^{-1} \abs{\partial_{x_i} Q_t g(x)}^2 \right) \\
		=& \, \frac{\gamma \tau_\kappa \abs{P_t f(x)}_*^2 + \gamma {\tau_\kappa}^{-1} \abs{Q_t g(x)}_*^2 }{2} - r(x)\kappa E_\kappa(u(x,t)) \\
		&+ r(x)B_\kappa(u(x,t)) - 2 \sum_{i=1}^d \partial_{\eta_i} B_\kappa(u(x,t)) Q_t g_i(x) \\
		\geqslant& \, \gamma \abs{P_t f(x)}_* \abs{Q_t g(x)}_* - \kappa r(x) E_\kappa(u(x,t)) \\
		&+ r(x)B_\kappa(u(x,t)) - 2 \sum_{i=1}^d \partial_{\eta_i} B_\kappa(u(x,t)) Q_t g_i(x).
		\end{align*}
		This proves \eqref{eq:1}.
		The next step is to show that
		\begin{equation} \label{eq:2}
		r(x)B(u(x,t)) - 2 \sum_{i=1}^d \partial_{\eta_i} B(u(x,t)) Q_t g_i(x) \geqslant 0.
		\end{equation}
		Note that \eqref{eq:2}, unlike \eqref{eq:1}, features function $B$ and not $B_\kappa$.
		
		We have the following equalities:
		\[
		\pdv{\beta}{y} (x,y) = q y^{q-1} + \gamma
		\begin{cases}
		(2-q) x^2 y^{1-q} \\
		(2-q) y^{q-1}
		\end{cases},
		\]
		
		\[
		\pdv{\abs{\eta}}{\eta_i} = \pdv{\sqrt{\eta_1^2 + \cdots + \eta_d^2}}{\eta_i} =  \frac{\eta_i}{\sqrt{\eta_1^2 + \cdots + \eta_d^2}} = \frac{\eta_i}{\abs{\eta}},
		\]
		
		\begin{align*}
		2\pdv{\eta_i} B(\zeta, \eta) &= \pdv{\eta_i} \beta(\abs{\zeta}, \abs{\eta}) = \pdv{\beta}{y} (\abs{\zeta}, \abs{\eta}) \cdot \pdv{\abs{\eta}}{\eta_i} \\
		&= \left( q \abs{\eta}^{q-1} + \gamma (2-q)
		\begin{cases}
		\abs{\zeta}^2 \abs{\eta}^{1-q} \\
		\abs{\eta}^{q-1}
		\end{cases}
		\right) \frac{\eta_i}{\abs{\eta}}.
		\end{align*}
		Using them, we may rewrite inequality \eqref{eq:2} as
		\begin{equation} \label{eq:34}
		\begin{aligned}
		\left( \abs{x}^2 + 2d \right)
		\left( \abs{\zeta}^p + \abs{\eta}^q + \gamma 
		\begin{cases}
		\abs{\zeta}^2 \abs{\eta}^{2-q} \\
		\frac{2}{p} \abs{\zeta}^p + \left( \frac{2}{q} - 1 \right) \abs{\eta}^q
		\end{cases}
		\right)
		-  \\
		2
		\left( q \abs{\eta}^{q} + \gamma (2-q)
		\begin{cases}
		\abs{\zeta}^2 \abs{\eta}^{2-q} \\
		\abs{\eta}^{q}
		\end{cases}
		\right) \geqslant 0,
		\end{aligned}
		\end{equation}
		where $\zeta = P_tf(x)$ and $\eta = Q_tg(x)$.
		Then, we consider two cases. \\
		\textit{Case 1:} $\abs{\zeta}^p \leqslant \abs{\eta}^q$. Observe that in order to prove \eqref{eq:34} it is sufficient to show that
		\[
		d \abs{\zeta}^p + (d-q) \abs{\eta}^q + \gamma (d-2+q) \abs{\zeta}^2 \abs{\eta}^{2-q} \geqslant 0.
		\]
		Since $q \leqslant 2$, this is true as long as $d \geqslant 2$. \\
		\textit{Case 2:} $\abs{\zeta}^p \geqslant \abs{\eta}^q$. In this case inequality \eqref{eq:34} becomes
		\[
		\left( \abs{x}^2 + 2d \right) \left( 1 + \frac{2\gamma}{p} \right)\abs{\zeta}^p + \left( \left( \abs{x}^2 + 2d \right) \left( 1 + \frac{2 \gamma}{q} - \gamma \right) -2q - 2\gamma (2-q) \right) \abs{\eta}^q \geqslant 0,
		\]
		but as before it is enough to prove that
		\[
		2d \left( 1 + \frac{2\gamma}{q} - \gamma \right) -2q -4\gamma + 2\gamma q \geqslant 0.
		\]
		Plugging the definition of $\gamma$ into this inequality and rearranging it, we arrive at
		\[
		q^3 + q^2 (-d-3) + q (3d - 6) + 6d \geqslant 0.
		\]
		This is equivalent to
		\[
		(d-q)(q^2 - 3q - 6) \leqslant 0,
		\]
		which is true for $1 < q \leqslant 2$ and $d \geqslant 2$.
		
		Having proved \eqref{eq:2}, we come back to \eqref{eq:1} and write
		\begin{equation} \label{eq:5}
		\begin{aligned}
		\left(\partial_t^2 + \Delta \right) (b_\kappa) (x,t) &\geqslant \, \gamma \abs{P_t f(x)}_* \abs{Q_t g(x)}_* - \kappa r(x) E_\kappa(u(x,t)) \\
		&+ r(x)B_\kappa(u(x,t)) - 2 \sum_{i=1}^d \partial_{\eta_i} B_\kappa(u(x,t)) Q_t g_i(x) \\
		&- r(x)B(u(x,t)) + 2 \sum_{i=1}^d \partial_{\eta_i} B(u(x,t)) Q_t g_i(x).
		\end{aligned}
		\end{equation}
		The last step is to show that 
		\[
		\kappa(n) r(x) E_{\kappa(n)} (u(x,t))
		\]
		on one hand, and the difference between 
		\[
		r(x)B(u(x,t)) - 2 \sum_{i=1}^d \partial_{\eta_i} B(u(x,t)) Q_t g_i(x)
		\]
		and
		\[
		r(x)B_{\kappa(n)}(u(x,t)) - 2 \sum_{i=1}^d \partial_{\eta_i} B_{\kappa(n)}(u(x,t)) Q_t g_i(x)
		\]
		on the other, vanish as $n \to \infty$.
		
		First let us prove that $u(x,t)$ is bounded on $X_n \times [0, +\infty)$. Recall that
		\[
		u(x,t) = \left( P_t f(x), Q_t g(x) \right) = \left( P_t f(x), Q_t g_1(x), \dots, Q_t g_d(x) \right),
		\]
		where
		\[
		P_t f = \sum_{n \in \N^d} e^{-t \lambda_n'^{1/2}} \innprod{f}{h_n} h_n, \quad Q_t g_i = \sum_{n \in \N^d} e^{-t (\lambda_n'-2)^{1/2}} \innprod{g_i}{h_n} h_n
		\]
		and $f, g_i \in \mathcal{D}$.
		Since $h_k$ are continuous, they are bounded on $X_n$, thus
		\[
		\abs{P_t f(x)} \leqslant \sum_{k \in \N^d} e^{-t \lambda_k'^{1/2}} \abs{\innprod{f}{h_k}} M_{n,k}
		\]
		for some constants $M_{n,k}$. The above sum has only finitely many non-zero terms and it is a decreasing function of $t$, so $P_t f(x)$ is bounded uniformly for all $x \in X_n$ and $t \geqslant 0$. A similar argument shows that each $Q_t g_i$ is bounded.
		
		Using inequality \ref{eq:bellman5}. from \thref{bellman} and the previous paragraph, we see that there exists a sequence $\kappa(n) \in (0, 1)$, such that for all $t \geqslant 0$ we have
		\begin{equation} \label{eq:k1}
		\int_{X_n} \abs{\kappa(n) r(x) E_{\kappa(n)} (u(x,t))} \, dx \leqslant \frac{1}{n}.
		\end{equation}
		
		Now we turn to estimating $\abs{B(u(x,t)) - B_\kappa (u(x,t))}$. As we have shown, $u[X_n \times [0, +\infty)]$ is bounded in $\R \times \R^d$, which means that $B$ is uniformly continuous on $u[X_n \times [0, +\infty)] + D(0, 1)$, where $D(0, 1)$ denotes the unit disc centered at 0. Therefore, refining $\kappa(n)$ if necessary, we can achieve
		\[
		\abs{B(u(x,t)) - B(u(x,t) - y)} \leqslant \frac{1}{n} \left( \int_{X_n} \abs{r(x)} \, dx \right)^{-1}
		\]
		for all $(x, t) \in X_n \times [0, +\infty)$ and $\abs{y} \leqslant \kappa(n)$. This in turn implies
		\begin{align} \label{eq:k2}
		\abs{B(u(x,t)) - B_{\kappa(n)} (u(x,t))} &\leqslant \int_{D(0, \kappa(n))} \abs{B(u(x,t)) - B(u(x,t) - y)} \psi_{\kappa(n)}(y) \, dy \nonumber \\
		&\leqslant \frac{1}{n} \left( \int_{X_n} \abs{r(x)} \, dx \right)^{-1}.
		\end{align}
		A similar reasoning shows that for each $n \in \N$ there exists $\kappa(n)$ satisfying \eqref{eq:k1} and \eqref{eq:k2} and such that for all $x \in X_n$, $t \geqslant 0$ and $i = 1, \dots, d$
		\begin{equation} \label{eq:k3}
		\abs{\partial_{\eta_i} B(u(x,t)) - \partial_{\eta_i} B_{\kappa(n)} (u(x,t))} \leqslant \frac{1}{n} \left( \int_{X_n} \abs{2 Q_t g_i(x)} \, dx \right)^{-1}.
		\end{equation}
		Coming back to inequality \eqref{eq:5}, we get
		\begin{align*}
		\int_{X_n} & \left( \partial_t^2 + \Delta \right) (b_{\kappa(n)})(x,t) \, dx \\
		&\geqslant \gamma \int_{X_n} \abs{P_t f(x)}_* \abs{Q_t g(x)}_* \, dx - \int_{X_n} \kappa(n) r(x) E_{\kappa(n)}(u(x,t)) \, dx \\
		&+ \int_{X_n} r(x)\left( B_{\kappa(n)}(u(x,t)) - B(u(x,t)) \right) \, dx \\
		&- 2 \int_{X_n} \sum_{i=1}^d Q_t g_i(x) \left( \partial_{\eta_i} B_{\kappa(n)}(u(x,t)) - \partial_{\eta_i} B(u(x,t)) \right) \, dx.
		\end{align*}
		Using conditions \eqref{eq:k1}, \eqref{eq:k2} and \eqref{eq:k3} on $\kappa(n)$ we get
		\[
		\liminf_{n \rightarrow \infty} \int_{X_n}  \left( \partial_t^2 + \Delta \right) (b_{\kappa(n)})(x,t) \, dx \geqslant \gamma \int_{\R^d} \abs{P_t f(x)}_* \abs{Q_t g(x)}_* \, dx
		\]
		and by the monotone convergence theorem
		\[
		\liminf_{\varepsilon \rightarrow 0^+} \liminf_{n \rightarrow \infty} I(n,\varepsilon) \geqslant \gamma \int_0^\infty \int_{\R^d} \abs{P_t f(x)}_* \abs{Q_t g(x)}_* \, dx \, t \, dt.
		\]
		
	\end{proof}
	
	\begin{lemma} \thlabel{limsup}
		For $f, g$ as in \thref{bet} we have
		\[
		\limsup_{\varepsilon \rightarrow 0^+} \limsup_{n \rightarrow \infty} I(n,\varepsilon) \leqslant \frac{1+\gamma}{2} \left( \norm{f}_p^p + \norm{g}_q^q \right).
		\]
	\end{lemma}
	
	\begin{proof}
		Denote
		\[
		I_1(n,\varepsilon) = \int_0^\infty \int_{X_n} \partial_t^2 \left(b_{\kappa(n)} \right) (x,t) \, dx \, t e^{-\varepsilon t} \, dt,
		\]
		\[
		I_2(n,\varepsilon) = \int_0^\infty \int_{X_n} \Delta \left(b_{\kappa(n)}\right) (x,t) \, dx \, t e^{-\varepsilon t} \, dt.
		\]
		Then $I(n, \varepsilon) = I_1(n, \varepsilon) + I_2(n, \varepsilon)$. First we prove that $\lim_{n \rightarrow \infty} I_2(n, \varepsilon) = 0$. Since
		\[
		I_2(n, \varepsilon) = \sum_{i=1}^d \int_0^\infty \int_{X_n} \partial_{x_i}^2 \left( b_{\kappa(n)} \right) (x,t) \, dx \, te^{-\varepsilon t} \, dt,
		\]
		it is sufficient to prove that each summand tends to $0$. We will present the proof for the first term only, call it $I_2^1(n, \varepsilon)$. Let $x' = (x_2, \dots, x_d)$. Integrating by parts with respect to $x_1$, we get
		\[
		I_2^1(n,\varepsilon) = \int_0^\infty \int_{[-n,n]^{d-1}} \left( \partial_{x_1} \left( b_{\kappa(n)} \right) (n,x',t) - \partial_{x_1} \left( b_{\kappa(n)} \right) (-n,x',t) \right) \, dx' \, te^{-\varepsilon t} \, dt.
		\]
		By the chain rule
		\begin{align*}
		\partial_{x_1} \left( b_{\kappa(n)} \right) (\pm n,x',t) &= \partial_\zeta B_{\kappa(n)} ( u(\pm n, x', t) ) \partial_{x_1} P_t f(\pm n, x') \\
		&+ \innprod{\nabla_\eta B_{\kappa(n)} ( u(\pm n, x', t) )}{\partial_{x_1} Q_t g(\pm n, x')}.
		\end{align*}
		Recall that $f, g_i \in \mathcal{D}$ and hence $P_tf, Q_t g_i \in \mathcal{D}$. Using item \ref{eq:bellman2}. of \thref{bellman} and the fact that the Hermite functions converge to 0 rapidly we conclude that $\lim_{n \rightarrow \infty} I_2(n, \varepsilon) = 0$.
		
		Now we turn to $I_1$. Using Fubini's theorem, we may interchange the order of integration to get
		\[
		I_1(n,\varepsilon) = \int_{X_n} \int_0^\infty \partial_t^2 \left(b_{\kappa(n)} \right) (x,t) \, t e^{-\varepsilon t} \, dt \, dx.
		\]
		Next, we use integration by parts on the inner integral twice, neglecting the boundary terms (this is allowed by the same argument as in the previous paragraph). This leads to
		\begin{align*}
		I_1(n, \varepsilon) &= -\int_{X_n} \int_0^\infty \partial_t \left(b_{\kappa(n)} \right) (x,t) \, (1-\varepsilon t) e^{-\varepsilon t} \, dt \, dx \\
		&= \int_{X_n} b_{\kappa(n)}(x,0) \, dx + \varepsilon^2 \int_{X_n} \int_0^\infty b_{\kappa(n)} (x,t) \, t e^{-\varepsilon t} \, dt \, dx \\
		&- 2\varepsilon \int_{X_n} \int_0^\infty b_{\kappa(n)} (x,t) \, e^{-\varepsilon t} \, dt \, dx \\
		&\leqslant \int_{X_n} b_{\kappa(n)}(x,0) \, dx + \varepsilon^2 \int_{X_n} \int_0^\infty b_{\kappa(n)} (x,t) \, t e^{-\varepsilon t} \, dt \, dx.
		\end{align*}
		Denote the last two terms by $I_1^1(n)$ and $I_1^2(n, \varepsilon)$.
		
		First we will show that $\limsup_{\varepsilon \rightarrow 0^+} \limsup_{n \rightarrow \infty} I_1^2(n,\varepsilon) = 0$. Before we proceed, we remind the reader that the semigroups $P_t$ and $Q_t$ were defined in \eqref{eq:semigroups}. Item \ref{eq:bellman1}. of \thref{bellman} implies that
		\[
		I_1^2(n,\varepsilon) \leqslant \varepsilon^2 C_p \int_{X_n} \int_0^\infty \left( \abs{P_t f(x)}^p + \abs{Q_t g(x)}^q + \max\left( \kappa(n)^p , \kappa(n)^q \right) \right) \, t e^{-\varepsilon t} \, dt \, dx.
		\]
		Taking $\kappa(n)$ satisfying \eqref{eq:k1}, \eqref{eq:k2} and \eqref{eq:k3} and such that 
		\begin{equation} \label{eq:k4}
		(2n)^d \max\left( \kappa(n)^p , \kappa(n)^q \right) \leqslant \frac{1}{n},
		\end{equation}
		we get
		\[
		\limsup_{n \rightarrow \infty} I_1^2(n,\varepsilon) \leqslant \varepsilon^2 C_p \int_{X} \int_0^\infty \left( \abs{P_t f(x)}^p + \abs{Q_t g(x)}^q\right) \, t \, dt \, dx \leqslant C \varepsilon^2.
		\]
		
		The last step is to estimate $I_1^1(n)$. Using item \ref{eq:bellman1}. of \thref{bellman} again, we obtain
		\[
		I_1^1(n) \leqslant \frac{1+\gamma}{2} \int_{X_n} \left( \abs{f(x)} + \kappa(n) \right)^p \, dx + \frac{1+\gamma}{2} \int_{X_n} \left( \abs{g(x)} + \kappa(n) \right)^q \, dx.
		\]
		We take $\varepsilon > 0$, denote $A = \{x \in \R^d: \varepsilon\abs{f(x)} \geqslant \abs{\kappa(n)}\}$ and split these two integrals as follows:
		\begin{align*}
		I_1^1(n) &\leqslant \frac{1+\gamma}{2} \int_{A} \left( \abs{f(x)} + \kappa(n) \right)^p \, dx + \int_{\complement{A}} \left( \abs{f(x)} + \kappa(n) \right)^p  \, dx \\
		&+ \frac{1+\gamma}{2} \int_{A} \left( \abs{g(x)} + \kappa(n) \right)^q \, dx + \int_{\complement{A}} \left( \abs{g(x)} + \kappa(n) \right)^q \, dx \\
		&\leqslant \frac{1+\gamma}{2} \left( (1+\varepsilon)^p \norm{f}_p^p + (1+\varepsilon)^q \norm{g}_q^q \right) \\
		&+ \frac{1+\gamma}{2}(2n)^d \left( \left(1+\varepsilon^{-1} \right)^p \kappa(n)^p + \left(1+\varepsilon^{-1} \right)^q \kappa(n)^q \right).
		\end{align*}
		Since $\kappa(n)$ satisfies \eqref{eq:k4}, we get
		\[
		\limsup_{\varepsilon \rightarrow 0^+} \limsup_{n \rightarrow \infty} I_1^1(n,\varepsilon) \leqslant \frac{1+\gamma}{2} \left( \norm{f}_p^p + \norm{g}_q^q \right)
		\]
		and hence, as we have shown that other terms are negligible, we obtain
		\[
		\limsup_{\varepsilon \rightarrow 0^+} \limsup_{n \rightarrow \infty} I(n,\varepsilon) \leqslant \frac{1+\gamma}{2} \left( \norm{f}_p^p + \norm{g}_q^q \right).
		\]
	\end{proof}
	
	Now we are ready to prove the bilinear embedding theorem.
	
	\begin{proof}[Proof of \thref{bet}]
		Combining \thref{liminf} and \thref{limsup}, we get
		\[
		\int_0^\infty \int_{\R^d} \abs{P_t f(x)}_* \abs{Q_t g(x)}_* \, dx \, t \, dt \leqslant \frac{1+\gamma}{2\gamma} \left( \norm{f}_p^p + \norm{g}_q^q \right).
		\]
		Multiplying $f$ by $\left( \frac{q\norm{g}_q^q}{p\norm{f}_p^p} \right)^{\frac{1}{p+q}}$ and $g$ by the reciprocal of this number, we obtain
		\[
		\int_0^\infty \int_{\R^d} \abs{P_t f(x)}_* \abs{Q_t g(x)}_* \, dx \, t \, dt \leqslant \frac{1+\gamma}{2\gamma} \left( \left( \frac{q}{p} \right)^{1/q} + \left( \frac{p}{q} \right)^{1/p} \right) \norm{f}_p \norm{g}_q.
		\]
		We need to show that $\frac{1+\gamma}{2\gamma} \left( \left( \frac{q}{p} \right)^{1/q} + \left( \frac{p}{q} \right)^{1/p} \right) \leqslant 6(p^*-1)$. Recall that $p \geqslant 2$, so $p^* = p$ and $1 < q \leqslant 2$, hence
		\begin{align*}
		\frac{1+\gamma}{2\gamma} \left( \left( \frac{q}{p} \right)^{1/q} + \left( \frac{p}{q} \right)^{1/p} \right) &= \frac{8+q(q-1)}{2}(q-1)^{\frac{1}{q}-1}(p-1) \\
		&\leqslant (q+3)(q-1)^{\frac{1}{q}-1}(p-1) \leqslant 6(p-1).
		\end{align*}
		A proof of the last inequality can be found in \cite[p.~761]{wrobel}. This proves the bilinear embedding theorem for $p \geqslant 2$.
		
		If $p \leqslant 2$, we switch $p$ and $q$ in the definition of $\beta$ and $\gamma$ in \eqref{eq:beta}, we switch $P_t f$ and $Q_t g$ in the definition of $u$ in \eqref{eq:u} and we consider function $B_\kappa$ from Section \ref{sec31} with $m_1 = d$ and $m_2 = 1$. Then we repeat the argument.
	\end{proof}
	
	\subsection{Proof of \texorpdfstring{\thref{th1}}{Theorem 2}} \label{sec33}
	
	Having proved the bilinear embedding theorem, we move on to the main result of this section.
	
	\begin{proof}		
		First we consider the case $d \geqslant 2$. By duality, it is sufficient to prove that
		\[
		\abs{\sum_{i=1}^d\innprod{R_i' f}{g_i}} \leqslant 36(p^*-1)\norm{f}_p \norm{\left( \sum_{i=1}^d \abs{g_i}^2 \right)^{1/2}}_q
		\]
		for any $f, g_i \in \mathcal{D}$. Since $\mathcal{D}$ is dense in $L^p$ for $1 \leqslant p < \infty$, this will mean that $\mathbf{R'}$ admits a bounded extension to the whole $L^p$ space with the same norm. By \thref{lem2}, we have
		\begin{align} \label{eq:th2}
		&\abs{\sum_{i=1}^d\innprod{R_i' f}{g_i}} \leqslant 4 \int_0^\infty \sum_{i=1}^d \abs{\innprod{\delta_i^* P_t f}{\partial_t Q_t g_i}} t \, dt \nonumber \\ 
		&\leqslant 4\int_0^\infty \int_{\R^d} \left( \sum_{i=1}^d \abs{\delta_i^* P_t f(x)}^2 \right)^{1/2} \left( \sum_{i=1}^d \abs{\partial_t Q_t g_i(x)}^2 \right)^{1/2} \, dx \, t \, dt \\
		&\leqslant 4\sqrt{2} \int_0^\infty \int_{\R^d} \abs{P_t f(x)}_* \abs{Q_t g(x)}_* \, dx \, t \, dt \leqslant 36(p^*-1)\norm{f}_p \norm{\left( \sum_{i=1}^d \abs{g_i}^2 \right)^{1/2}}_q \nonumber
		\end{align}
		The last inequality follows from \thref{bet}.
		
		Now assume that $d = 1$. We will show that $\mathbf{R}'= R_1' = \delta_1^* L'^{-1/2}$ is the adjoint of $\mathbf{R} = R_1 = \delta_1 L^{-1/2}$. Since $\mathbf{R}$ is the operator considered in \cite[Section 5.4]{wrobel}, we can then use \cite[Theorem 9]{wrobel} to get the desired result. That theorem features the constant 48, but we can refine the proof of \cite[Theorem 1]{wrobel}, which is the main ingredient in the proof of \cite[Theorem 9]{wrobel}, in a manner similar to \eqref{eq:th2} to obtain the constant 36.
		
		To prove the adjointness, we check that $\innprod{h_n}{R_1' h_k} = \innprod{R_1 h_n}{h_k}$. For the left-hand side we use \eqref{eq:l'} and item \ref{eq:fact2}. from \thref{fact}.
		
		\begin{equation} \label{eq:adj1}
		\begin{aligned}
		\innprod{h_n}{R_1' h_k} &= \innprod{h_n}{\delta_1^* (L+2)^{-1/2} h_k} = (\lambda_k+2)^{-1/2} \innprod{h_n}{\delta_1^* h_k} \\
		&= \sqrt{2(k_1+1)} (\lambda_k+2)^{-1/2} \innprod{h_n}{h_{k+e_1}} \\
		&=  \begin{cases}
		\sqrt{\frac{2(k_1+1)}{2\abs{k}+d+2}} \quad &\text{if } n = k + e_1 \\
		0 \quad &\text{otherwise}
		\end{cases}.
		\end{aligned}
		\end{equation}
		
		For the right-hand side we use item \ref{eq:fact1}.
		
		\begin{equation} \label{eq:adj2}
		\begin{aligned}
		\innprod{R_1 h_n}{h_k} &= \innprod{\delta_1 L^{-1/2} h_n}{h_k} = \lambda_n^{-1/2} \innprod{\delta_1 h_n}{h_k} \\
		&= \sqrt{2n_1} \lambda_n^{-1/2} \innprod{h_{n-e_1}}{h_k} \\
		&=  \begin{cases}
		\sqrt{\frac{2n_1}{2\abs{n}+d}} \quad &\text{if } n - e_1 = k \\
		0 \quad &\text{otherwise}
		\end{cases}.
		\end{aligned}
		\end{equation}
		This completes the proof.
	\end{proof}
	
	\section{Riesz transforms of the second kind} \label{sec4}
	
	This section is devoted to estimating the norm of the vector of the Riesz transforms
	\[
	\tilde{R}_i f(x) = \delta_i^* L^{-1/2} f(x).
	\]
	As noted earlier, we will give a result similar to \cite[Corollary 1]{dragicevic_volberg} but with an explicit constant.
	
	We want to estimate
	\[
	\norm{\mathbf{\tilde{R}} f}_p \coloneqq \left( \int_{\R^d} \abs{\mathbf{\tilde{R}}f(x)}^p \, dx \right)^{1/p}.
	\]
	Observe that for $f \in \mathcal{D}$ it holds
	\begin{align*}
	\tilde{R}_i f(x) &= \delta_i^*L^{-1/2} f(x) = \left( -\partial_{x_i} + x_i \right) L^{-1/2}f(x) \\
	&= -\delta_i L^{-1/2}f(x) + 2x_i L^{-1/2}f(x) \\
	&= R_i^1f(x) + R_i^2f(x).
	\end{align*}
	Then $\mathbf{\tilde{R}}f(x) = \mathbf{R^1}f(x) + \mathbf{R^2}f(x)$ (with $\mathbf{\tilde{R}}f(x) = \left(\tilde{R}_1 f(x), \dots, \tilde{R}_d f(x) \right)$ and $\mathbf{R^1}$ and $\mathbf{R^2}$ defined analogously), hence
	\[
	\abs{\mathbf{\tilde{R}}f(x)} \leqslant \abs{\mathbf{R^1}f(x)} + \abs{\mathbf{R^2}f(x)}
	\] 
	and
	\begin{equation} \label{eq:R12}
	\norm{\mathbf{\tilde{R}}f}_p \leqslant \norm{\mathbf{R^1}f}_p + \norm{\mathbf{R^2}f}_p.
	\end{equation}
	\cite[Theorem 9]{wrobel} gives the bound of $48(p^*-1)$ (which, as mentioned in Section \ref{sec33}, can be reduced to $36(p^*-1)$) for the $L^p$ norm of $\mathbf{R^1}$, so we will focus on $\mathbf{R^2}$. Next, note that
	\[
	\abs{\mathbf{R^2}f(x)} = 2 \left( \sum_{i=1}^d \abs{x_iL^{-1/2}f(x)}^2 \right)^{1/2}= 2\abs{x}\abs{L^{-1/2}f(x)},
	\]
	which means that it is sufficient to deal with the operator $\abs{x}L^{-1/2}$, formally defined on $\mathcal{D}$ as
	\[
	Sf(x) = \abs{x} L^{-1/2}f(x).
	\]
	This operator turns out to be bounded on all $L^p$ spaces for $1 \leqslant p < \infty$.
	
	\begin{theorem} \thlabel{th2}
		For $1 \leqslant p < \infty$ we have $\norm{S}_{p \rightarrow p} \leqslant 3$.
	\end{theorem}
	
	In order to prove this theorem, we first derive an expression for the kernel of $S$, i.e., a function $K(x,y)$ such that
	\[
	Sf(x) = \int_{\R^d} K(x,y)f(y) \, dy \quad \text{for } f \in \mathcal{D}.
	\]
	
	\begin{lemma} \thlabel{lem8}
		For $x,y \in \R^d$ we have
		\[
		K(x,y) = \abs{x} \int_0^\infty \frac{1}{\sqrt{t}} K_t(x,y) \, dt,
		\]
		where
		\begin{equation} \label{eq:kt}
		K_t(x,y) = \frac{C_d}{(\sinh 2t)^{d/2}} \exp \left( -\frac{\abs{x-y}^2}{4\tanh t} - \frac{\tanh t}{4} \abs{x+y}^2 \right), \quad C_d = \frac{1}{(2 \pi)^{d/2}\sqrt{\pi}}.
		\end{equation}
	\end{lemma}
	
	\begin{proof}
		Identity \cite[4.1.2]{thangavelu} states that
		\[
		e^{-tL}f(x) = \frac{1}{(2\pi)^{d/2}} \int_{\R^d} K_t'(x,y)f(y) \, dy,
		\]
		with
		\begin{align*}
		K_t'(x,y) &= \frac{1}{(\sinh 2t)^{d/2}} \exp \left( -\frac{\abs{x}^2+\abs{y}^2}{2}\coth 2t + \frac{\innprod{x}{y}}{\sinh 2t} \right) \\
		&= \frac{1}{(\sinh 2t)^{d/2}} \exp \left( -\frac{\abs{x-y}^2}{4\tanh t} - \frac{\tanh t}{4} \abs{x+y}^2 \right).
		\end{align*}
		Note also that
		\[
		\lambda^{-1/2} = \frac{1}{\sqrt{\pi}} \int_0^\infty e^{-t\lambda} \frac{1}{\sqrt{t}} \, dt.
		\]
		Since $\mathcal{D} = \lin \{ h_n: n \in \N^d \}$, it is sufficient to prove the formula for $f = h_n$. We have
		\begin{align*}
		L^{-1/2}h_n(x) &= \lambda_n^{-1/2} h_n(x) = \frac{1}{\sqrt{\pi}} \int_0^\infty e^{-t\lambda_n} h_n(x) \frac{1}{\sqrt{t}} \, dt \\
		&= \frac{1}{\sqrt{\pi}} \int_0^\infty e^{-tL} h_n(x) \frac{1}{\sqrt{t}} \, dt \\
		&= \frac{1}{\sqrt{\pi}}\frac{1}{(2\pi)^{d/2}} \int_0^\infty \frac{1}{\sqrt{t}} \int_{\R^d} K_t'(x,y)h_n(y) \, dy \, dt.
		\end{align*}
		This integral is absolutely convergent, so we may interchange the order of integration and the conclusion follows.
	\end{proof}
	
	Next we prove that the operator $T$ defined on $L^p$, $1 \leqslant p \leqslant \infty$, as 
	\[
	Tf(x) = \int_{\R^d} K(x,y)f(y) \, dy
	\]
	is bounded uniformly in $d$ and $p$. This will mean that $S$ is bounded on $\mathcal{D}$ in $L^p$ norm and, by density, that it has a unique bounded extension to $L^p$ for $1 \leqslant p < \infty$ with the same norm. We want to use interpolation and our goal is to prove that
	\begin{equation} \label{eq1}
	\int_{\R^d} K(x,z) \, dz \leqslant 2 \quad \text{and} \quad \int_{\R^d} K(z,y) \, dz \leqslant 3
	\end{equation}
	for all $x,y \in \R^d$. Clearly, we have
	
	\begin{equation} \label{eqK}
	\begin{aligned}
	\int_{\R^d} K(z,y) \, dz &= \int_{\R^d} \abs{z} \int_0^\infty \frac{1}{\sqrt{t}} K_t(z,y) \, dt \, dz \\
	&\leqslant \int_{\R^d} \abs{y} \int_0^\infty \frac{1}{\sqrt{t}} K_t(z,y) \, dt \, dz \\
	&+ \int_{\R^d} \abs{y-z} \int_0^\infty \frac{1}{\sqrt{t}} K_t(z,y) \, dt \, dz \\
	&= \int_{\R^d} K(y, z) \, dz +  \int_{\R^d} \abs{y-z} \int_0^\infty \frac{1}{\sqrt{t}} K_t(z,y) \, dt \, dz,
	\end{aligned}
	\end{equation}
	so, by symmetry of $K_t$, it is sufficient to prove the first inequality of \eqref{eq1} and the following proposition.
	
	\begin{proposition} \thlabel{prop1}
		For $y \in \R^d$ it holds
		\begin{equation} \label{eq2}
		\int_{\R^d} \abs{y-z} \int_0^\infty \frac{1}{\sqrt{t}} K_t(z,y) \, dt \, dz \leqslant 1.
		\end{equation}	
	\end{proposition}
	
	\begin{proof}
		
		We begin with an auxiliary computation:
		\begin{equation} \label{eq:comp}
		I(k) \coloneqq \int_{\R^d} \abs{x} e^{-k \abs{x}^2} \, dx = \frac{\Gamma \left( \frac{d+1}{2} \right)}{\Gamma \left( \frac{d}{2} \right)} \frac{\pi^{d/2}}{k^{(d+1)/2}} \quad \text{for } k > 0.
		\end{equation}
		
		To prove \eqref{eq:comp}, let $S_d = \frac{2 \pi^{d/2}}{\Gamma \left( \frac{d}{2} \right)}$ denote the surface area of the unit sphere in the $d$-dimensional Euclidean space. Then we can write
		\begin{align*}
		\int_{\R^d} \abs{x} e^{-k\abs{x}^2} \, dx &= \int_0^\infty r e^{-kr^2} r^{d-1} S_d \, dr = \frac{S_d}{2k^{(d+1)/2}} \int_0^\infty t^{(d-1)/2} e^{-t} \, dt \\
		&= \frac{\Gamma \left( \frac{d+1}{2} \right)}{\Gamma \left( \frac{d}{2} \right)} \frac{\pi^{d/2}}{k^{(d+1)/2}}.
		\end{align*}

		Coming back to \eqref{eq2}, in view of \eqref{eq:comp} we have, for $t \geqslant 0$,
		\begin{align*}
		\int_{\R^d} \abs{x-y} K_t(x,y) \, dx &= \frac{C_d}{(\sinh 2t)^{d/2}} \int_{\R^d} \abs{x-y} \exp \left( -\frac{\abs{x-y}^2}{4\tanh t} - \frac{\tanh t}{4} \abs{x+y}^2 \right) \, dx \\
		&\leqslant \frac{C_d}{(\sinh 2t)^{d/2}} \int_{\R^d} \abs{x-y} \exp \left( -\frac{\abs{x-y}^2}{4\tanh t} \right) \, dx \\
		&= \frac{C_d}{(\sinh 2t)^{d/2}} \int_{\R^d} \abs{x} \exp \left( -\frac{\abs{x}^2}{4\tanh t} \right) \, dx \\
		&= \frac{C_d}{(\sinh 2t)^{d/2}} I\left( \frac{1}{4 \tanh t} \right) \\
		&= \frac{\pi^{d/2}}{(2\pi)^{d/2}\sqrt{\pi}} \frac{\Gamma \left( \frac{d+1}{2} \right)}{\Gamma \left( \frac{d}{2} \right)} \frac{\left( 4 \tanh t \right)^{(d+1)/2}}{(\sinh 2t)^{d/2}} \\
		&= \frac{2}{\sqrt{\pi}} \frac{\Gamma \left( \frac{d+1}{2} \right)}{\Gamma \left( \frac{d}{2} \right)} \frac{\sqrt{\tanh t}}{(\cosh t)^d}
		\end{align*}
		
		Plugging it into \eqref{eq2}, we get
		
		\begin{align*}
		\int_{\R^d} \abs{y-z} \int_0^\infty \frac{1}{\sqrt{t}} K_t(z,y) \, dt \, dz &\leqslant \frac{2}{\sqrt{\pi}} \frac{ \Gamma \left( \frac{d+1}{2} \right)}{\Gamma \left( \frac{d}{2} \right)} \int_0^\infty \frac{\sqrt{\tanh t}}{(\cosh t)^d} \frac{dt}{\sqrt{t}} \\
		&\leqslant \frac{2}{\sqrt{\pi}} \frac{ \Gamma \left( \frac{d+1}{2} \right)}{\Gamma \left( \frac{d}{2} \right)} \int_0^\infty \frac{1}{(\cosh t)^d} dt
		\end{align*}
		
		To estimate the last integral, we will use formula \cite[5.12.7]{nist}:
		\[
		\int_0^\infty \frac{1}{\left( \cosh t \right)^{2a}} dt = 4^{a-1} \text{B}(a,a),
		\]
		where B denotes the beta function. We obtain
		\[
		\int_0^\infty \frac{1}{\left( \cosh t \right)^d} \, dt = 4^{\frac{d}{2}-1} \text{B}\left(\frac{d}{2}, \frac{d}{2}\right) = 2^{d-2} \frac{\Gamma \left( \frac{d}{2} \right)^2}{\Gamma(d)}.
		\]
		Finally, using the Legendre duplication formula ($\Gamma(z) \Gamma(z+\frac{1}{2}) = 2^{1-2z} \sqrt{\pi} \Gamma(2z)$), we get
		\begin{align*}
		&\int_{\R^d} \abs{y-z} \int_0^\infty \frac{1}{\sqrt{t}} K_t(z,y) \, dt \, dz \\ 
		&\leqslant \frac{2^{d-1}}{\sqrt{\pi}} \frac{\Gamma \left( \frac{d+1}{2} \right)}{\Gamma \left( \frac{d}{2} \right)} \frac{\Gamma \left( \frac{d}{2} \right)^2}{\Gamma(d)} 
		= 2^{d-1} \frac{\Gamma \left( \frac{d+1}{2} \right) \Gamma \left( \frac{d}{2} \right)}{\sqrt{\pi} \Gamma(d)} = 1.
		\end{align*}
	\end{proof}
	
	Now it remains to justify the first inequality of \eqref{eq1}.
	
	\begin{proposition} \thlabel{prop2}
		For $x \in \R^d$ we have
		\[
		\int_{\R^d} \abs{x} \int_0^\infty \frac{1}{\sqrt{t}} K_t(x,y) \, dt \, dy \leqslant \frac{1}{\sqrt{\pi}} + \sqrt{2}.
		\]
	\end{proposition}
	
	\begin{proof}
		Identity \cite[(4.4)]{dragicevic_volberg} gives
		\[
		\int_{\R^d} K_t(x, y) \, dy = \frac{1}{\sqrt{\pi}} \frac{1}{(\cosh 2t)^{d/2}} \exp\left( -\frac{\abs{x}^2}{2 \coth 2t} \right).
		\]
		
		To estimate the integral with respect to $t$, we need to split it into two parts. Let $\tau \in [0.95, 0.96]$ denote the unique positive solution of $2\coth(2t) = \frac{2}{t}$. It follows that $2\coth(2t) \leqslant \frac{2}{t}$ for $0 \leqslant t \leqslant \tau$. Thus, we obtain
		\begin{equation} \label{eq:45}
		\begin{aligned}
		&\frac{\abs{x}}{\sqrt{\pi}} \int_0^\tau \frac{1}{(\cosh 2t)^{d/2}} \exp\left( -\frac{\abs{x}^2}{2 \coth 2t} \right) \frac{dt}{\sqrt{t}} \leqslant \\
		&\frac{\abs{x}}{\sqrt{\pi}} \int_0^\tau \exp\left( -\frac{t \abs{x}^2}{2} \right) \frac{dt}{\sqrt{t}} \leqslant \frac{\abs{x}}{\sqrt{\pi}} \sqrt{\frac{2 \pi}{\abs{x}^2}} = \sqrt{2}
		\end{aligned}
		\end{equation}
		For the second part, when $t \geqslant \tau$ and $2\coth(2t) \leqslant \frac{2}{\tau}$, calculations are as follows:
		\begin{equation} \label{eq:46}
		\begin{aligned}
		&\frac{\abs{x}}{\sqrt{\pi}} \int_\tau^\infty \frac{1}{(\cosh 2t)^{d/2}} \exp\left( -\frac{\abs{x}^2}{2 \coth 2t} \right) \frac{dt}{\sqrt{t}} \leqslant \\
		&\frac{2^{d/2}\abs{x}}{\sqrt{\pi}} \exp\left( -\frac{\tau \abs{x}^2}{2} \right) \int_\tau^\infty e^{-td} \frac{dt}{\sqrt{t}} \leqslant \\
		&\frac{2^{d/2}e^{-\tau d}}{d \sqrt{\tau \pi}} \abs{x}\exp\left( -\frac{\tau \abs{x}^2}{2} \right) \leqslant \frac{1}{d \tau \sqrt{\pi e}} \leqslant \frac{1}{\sqrt{\pi}}
		\end{aligned}
		\end{equation}
		In the first inequality we used the fact that $\cosh(2t) \geqslant \frac{e^{2t}}{2}$. Combining \eqref{eq:45} and \eqref{eq:46} completes the proof.
	\end{proof}
	
	Now we are ready to prove the main theorem of this section.
	
	\begin{proof}[Proof of \thref{th2}]
		\thref{prop1}, \thref{prop2} and \eqref{eqK} imply that
		\[
		\int_{\R^d} K(x,z) \, dz \leqslant 3 \quad \text{and} \quad \int_{\R^d} K(z,y) \, dz \leqslant 3,
		\]
		hence $T$ is bounded on $L^1$ and $L^\infty$ with norm at most 3. Using the Riesz--Thorin interpolation theorem we obtain $\norm{T}_{p \rightarrow p} \leqslant 3$ for $1 \leqslant p \leqslant \infty$ and since $S = T$ on $\mathcal{D}$ --- a dense subspace of $L^p$ for $1 \leqslant p < \infty$ --- $S$ has a unique bounded extension to $L^p$ with norm at most 3.
	\end{proof}
	
	Recollecting \eqref{eq:R12}, we see that \thref{th2} and \cite[Theorem 9]{wrobel} imply an $L^p$ norm estimate for $\mathbf{\tilde{R}} f = \left( \tilde{R}_1 f, \dots, \tilde{R}_d f \right)$.
	
	\begin{theorem} \thlabel{th3}
		For $f \in L^p$ we have
		\[
		\norm{\mathbf{\tilde{R}} f}_p = \left( \int_{\R^d} \abs{\mathbf{\tilde{R}}f(x)}^p \, dx \right)^{1/p} \leqslant 42 (p^* - 1) \norm{f}_p.
		\]
	\end{theorem}
	
	As a corollary of the above result we will prove one more theorem. Let
	\[
	\mathbf{R^*} f = \left( R_1^* f, \dots, R_d^* f \right)
	\]
	with
	\[
	R_i^* f(x) = \delta_i^* (L+2)^{-1/2}f(x).
	\]
	It is worth noting that each $R_i^*$ is the adjoint of $R_i = \delta_i L^{-1/2}$ --- the 'usual' Riesz--Hermite transform. This fact has been already proved in Section \ref{sec33}, the relevant identities are \eqref{eq:adj1} and \eqref{eq:adj2}.
	
	Now we are ready to state the last theorem of this paper.
	
	\begin{theorem} \thlabel{th4}
		For $f \in L^p$ we have
		\[
		\norm{\mathbf{R^*} f}_p = \left( \int_{\R^d} \abs{\mathbf{R^*}f(x)}^p \, dx \right)^{1/p} \leqslant 84 (p^* - 1) \norm{f}_p.
		\]
	\end{theorem}
	
	To prove this theorem, we perform a slightly more general calculation. For $a > 0$ we define
	\[
	U_a f(x) = \left( L(L+2a)^{-1} \right)^{1/2} f(x), \quad  f \in \mathcal{D}.
	\]
	
	\begin{proposition} \thlabel{propU}
		For $1 \leqslant p < \infty$ we have $\norm{U_a}_{p \rightarrow p} \leqslant 2$.
	\end{proposition}
	
	\begin{proof}
		We begin with a fact: If $A$ is an operator with $\norm{A} \leqslant 1$, then
		\begin{equation} \label{eq:sqrt}
		\left( I - A \right)^{1/2} = I - \sum_{n=1}^\infty c_n A^n,
		\end{equation}
		where
		\[
		c_n = \frac{\left( 2n \right)!}{\left(n!\right)^2(2n-1) 4^n} \quad \text{and} \quad \sum_{n=1}^\infty c_n = 1.
		\]
		To prove \eqref{eq:sqrt}, note that 
		\begin{equation} \label{eq:sqrt2}
		\sqrt{1-x} = 1 - \sum_{n=1}^\infty c_n x^n \quad \text{for } \abs{x} \leqslant 1,
		\end{equation}
		with $c_n$ defined above and the series converges absolutely, hence the series in \eqref{eq:sqrt} converges for $\norm{A} \leqslant 1$. Squaring identities \eqref{eq:sqrt} and \eqref{eq:sqrt2} and comparing coefficients yields the result.
		
		Next, observe that
		\[
			L(L+2a)^{-1} = I - 2a (L+2a)^{-1},
		\]
		so, taking $A = 2a \left(L + 2a \right)^{-1}$ in \eqref{eq:sqrt}, we see that it is enough to prove that ${\norm{\left( L+2a \right)^{-1}}_{p \rightarrow p} \leqslant \frac{1}{2a}}$. We proceed as in the proof of \thref{th2}. First, we find the kernel of $\left(L+2a\right)^{-1}$, then prove its boundedness on $L^1$ and $L^\infty$ and finally use interpolation.
		
		A computation similar to the proof of \thref{lem8} shows that
		\[
		\left(L+2a\right)^{-1} f(x) = \int_{\R^d} \tilde{K}(x,y) f(y) \, dy \quad \text{for } f \in D,
		\]
		where
		\[
		\tilde{K}(x,y) = \int_0^\infty e^{-2at} \sqrt{\pi}{K}_t(x,y) \, dt
		\]
		and $K_t$ is defined as in \eqref{eq:kt}.
		Since this time the kernel is symmetric, we only prove that
		\[
		\int_{\R^d} \tilde{K}(x,y) \, dy \leqslant \frac{1}{2a}.
		\]
		From \cite[(4.4)]{dragicevic_volberg} and the definition of $\tilde{K}$ we get
		\[
		\int_{\R^d} \tilde{K}(x,y) \, dy \leqslant \int_0^\infty e^{-2at} \, dt = \frac{1}{2a}.
		\]
		This means that the operator $V$ defined as
		\[
		V f(x) = \int_{\R^d} \tilde{K}(x,y) f(y) \, dy
		\]
		is bounded on $L^1$ and $L^\infty$ with norm at most $\frac{1}{2a}$ and the Riesz--Thorin interpolation theorem gives its boundedness on $L^p$ for $1 \leqslant p \leqslant \infty$ with the same upper bound for the norm. Density of $\mathcal{D}$ implies that $\left(L+2a\right)^{-1}$ has a unique bounded extension to the whole $L^p$ space, $1 \leqslant p < \infty$, with norm at most $\frac{1}{2a}$. Applying \eqref{eq:sqrt} with $A = 2a \left(L + 2a \right)^{-1}$ completes the proof.
	\end{proof}
	
	This leads us to the proof of \thref{th4}.
	
	\begin{proof}[Proof of \thref{th4}]
		It is sufficient to note that for $f \in \mathcal{D}$
		\[
		R_i^* f = \delta_i^* \left(L+2\right)^{-1/2} f = \delta_i^* L^{-1/2} \left( L(L+2)^{-1} \right)^{1/2} f = \tilde{R}_i U_1 f.
		\]
		Now \thref{th3} and \thref{propU} complete the proof.
	\end{proof}
	
	Finally, let us mention that in the light of \eqref{eq:l'}, a very similar argument (with $U_d$ instead of $U_1$) can be used to prove \thref{th1} with the constant equal to 84.
	
	\section*{Acknowledgements}
	
	The author is very grateful to Błażej Wróbel for suggesting the topic, supervision and helpful discussions.
	
	The author is also very indebted to the reviewers for careful reading of the manuscript and many valuable comments.
	
	Research was supported by the National Science Centre, Poland, research project No. 2018/31/B/ST1/00204.
	
	The paper constitutes author's master's thesis.

\end{document}